\documentclass[12pt]{amsart}
\usepackage{amsmath,amsthm,amssymb,amsfonts, esint}
\usepackage[nobysame,abbrev,alphabetic]{amsrefs}
\usepackage{url}
\usepackage{ esint }
\usepackage{graphicx}
\usepackage{graphicx}
\usepackage{color}
\usepackage{times}
\usepackage{cite}
\usepackage{enumerate,latexsym}
\usepackage{latexsym}
\usepackage{amsmath,amssymb}
\usepackage{graphicx}
\usepackage{amsthm}
\usepackage{verbatim}
\newtheorem{thm}{Theorem}[section]
\newtheorem*{theorem*}{Theorem}
\newtheorem*{acknowledgement*}{Acknowledgement}
\newtheorem{cor}[thm]{Corollary}

\newtheorem{lem}[thm]{Lemma}

\theoremstyle{definition}
\newtheorem{defn}[thm]{Definition}
\theoremstyle{remark}
\newtheorem{rem}[thm]{Remark}

\numberwithin{equation}{section}

\newcommand{\dstyle}{\displaystyle}

\newcommand{\hide}[1]{}

\title[Integrability of scalar curvature and normal metric]{Integrability of scalar curvature and normal metric on conformally flat manifolds}
\date{May 7, 2017}

\author[Shengwen Wang]{Shengwen Wang}
\address{Shengwen Wang, Department of Mathematics, Johns Hopkins University,
3400 N. Charles Street, Baltimore, MD 21218}
\address{ email: swang@math.jhu.edu}
\thanks{The second author is partially supported
by NSF grant DMS-1612015}

\author[Yi Wang]{Yi Wang}
\address{Yi Wang, Department of Mathematics, Johns Hopkins University, 3400 N. Charles Street, Baltimore, MD 21218,}
\address{ email: ywang@math.jhu.edu}


\begin{document}

\maketitle

\begin{abstract}
On a manifold $(\mathbb{R}^n, e^{2u} |dx|^2)$, we say $u$ is normal if the $Q$-curvature equation that $u$ satisfies
$(-\Delta)^{\frac{n}{2}} u = Q_g e^{nu}$ can be written as the integral form $u(x)=\frac{1}{c_n}\int_{\mathbb R^n}\log\frac{|y|}{|x-y|}Q_g(y)e^{nu(y)}dy+C$. 
In this paper, we show that the integrability assumption on the negative part of the scalar curvature implies the metric is normal. As an application, we prove a bi-Lipschitz equivalence theorem for conformally flat metrics. 

\end{abstract}

\section{Introduction}

The $Q$-curvature arises naturally as a conformal invariant associated to
the Paneitz operator. When $n=4$, the Paneitz operator is defined as:
$$P_g=\Delta_g^2+\delta(\frac{2}{3}R_g\,g-2\, {\rm Ric_g})d,$$
where $\delta$ is the divergence operator, $d$ is the differential operator, $R$ is the scalar curvature of $g$, and ${\rm Ric}$
is the Ricci curvature tensor. The Branson's $Q$-curvature \cite{Branson} is defined as
$$Q_g=\frac{1}{12}\left\{-\Delta_g R_g +\frac{1}{4}R_g^2 -3|E_g|^2 \right\},
$$
where $E_g$ is the traceless part of ${\rm Ric_g}$, and $|\,\cdot\,|$ is the point-wise norm taken with respect to the metric $g$.
Under the conformal change $g=e^{2u}g_0$, the Paneitz operator transforms by $P_{g}=e^{-4u}P_{g_0}$,
and
$Q_{g}$ satisfies the fourth order equation
\begin{equation}\label{Paneitz4}
P_{g_{0}}u+2Q_{g_0}=2Q_{g}e^{4u}.\end{equation}
This is analogous to the transformation law satisfied by the Laplacian operator $-\Delta_g$ and the Gaussian curvature $K_g$ on surfaces,
\begin{equation}-\Delta_{g_0}u+K_{g_0}=K_{g}e^{2u}.\end{equation}

When the background metric $g_0$ is the flat metric $|dx|^2$, the transformation law \eqref{Paneitz4} that the $Q$-curvature satisfies becomes
\begin{equation}
\Delta^2_{g_0} u =2Q_{g}e^{4u}.\end{equation}

The invariance of the integration of the $Q$-curvature in dimension $4$ is due to the Gauss-Bonnet-Chern formula for a closed manifold $M$:
\begin{equation}\label{GBCEq}\chi(M)=\dstyle\frac{1}{4\pi^2} \int_{M}\left(\frac{|W_g|^2}{8}+Q_g\right) dv_g,\end{equation}
where $W_g$ denotes the Weyl tensor. For complete manifolds with conformally flat ends, the work of Chang, Qing, and Yang \cite{CQY1} proves the formula between the asymptotic isoperimetric ratio and the integration of the $Q$-curvature.

In Chang, Qing and Yang's work \cite{CQY1, CQY2}, they used an important notion ``normal metric" on conformally flat manifolds to prove the formula of the asymptotic isoperimetric ratio. Normal metric was first introduced by Huber \cite{Huber}, and later used by Finn \cite{Finn} and \cite{Hartman}. Huber proved that in dimension two, for a surface with finite total Gauss curvature, the metric is always normal. In \cite{CQY1}'s work, it is a key observation that if the scalar curvature at infinity is nonnegative, then the metric is normal. The proof mostly uses maximum principle and properties of harmonic functions.
In this paper, we generalize this result. We show that if the negative part of the scalar curvature is integrable, then the metric is normal. The main result is the following theorem. 

\begin{thm}\label{thm:main}
Let $(M^n,g)=(\mathbb R^n,e^{2u} |dx|^2)$ be a noncompact complete conformally flat metric of even dimension, satisfying  
\begin{equation}\int_{M^n}|Q_g|d v_g<\infty, \end{equation}
and
\begin{equation}\label{totalR}
\int_{\mathbb{R}^n}(R_{g}^-)^{\frac{n}{2}}dv_g<\infty,
\end{equation}
where $R_{g}^-$ is the negative part of the scalar curvature. 
Then the metric is normal.
\end{thm}

Therefore, as a direct corollary, we generalize Chang, Qing and Yang's work to the following theorem.

\begin{cor}
Let $(M^n,g)=(\mathbb{R}^n, e^{2w} |dx|^2)$ be a noncompact complete conformally flat manifold
of even dimension, satisfying 
\begin{equation}\int_{M^n}|Q_g|d v_g<\infty, \end{equation}
and
\begin{equation}\label{scalar}
\int_{M^n} (R_g^-)^{\frac{n}{2}} dv_g\leq \infty.
\end{equation} Then
\begin{equation}\label{totalQ}
\dstyle \frac{1}{c_n}\int_{M^n} Q_g dv_g\leq  \chi(\mathbb{R}^{n} )=1.
\end{equation}

Moreover, the difference of the two sides in the above inequality is given by the asymptotic isoperimetric ratio:
 \begin{equation}
\displaystyle \chi(\mathbb{R}^{n} )-\frac{1}{c_n}\int_{\mathbb{R}^n} Q_g dv_g=\lim_{r\rightarrow
\infty } \frac{{\rm Vol}_g(\partial B_{j}(r))^{n/(n-1)}}{n\omega_n^{\frac{1}{n}} \cdot {\rm Vol}_g(B_{j}(r))}.
\end{equation}
\noindent Here $B_{j}(r)$ denotes the Euclidean ball with radius $r$ at the $j$-th end.
\end{cor}


Hereafter $c_n$ denotes the constant $2 ^ {n -2} (\frac {n -2} 2)! \pi ^ {\frac n 2}$. It is the value of the integral of the $Q$-curvature on the unit $n$-hemisphere $\mathbb S^n_+$. $\omega_n$ denotes the volume of unit ball in $\mathbb R^n$.

An orientation preserving homeomorphism $f: \mathbb R^n\rightarrow \mathbb R^n$ is call quasiconformal map if 
$$ \sup_{x \in \mathbb R^n} H(x,f):= \limsup_{r\rightarrow 0+} \sup_{|u-x|=|v-x|=r }\frac{ |f(u)-f(x)|}{|f(v)-f(x)|} <\infty.
$$
We denote the dilation of a quasiconformal map by 
$$H(f):= \sup_{x\in \mathbb R^n} H(x,f)<\infty.$$ We call $f$ is $H$-quasiconformal if $H(f)\leq H$.

In \cite{YW1}, the second author has found a relation between the integral of the $Q$-curvature and the quasiconformal equivalence of two manifolds (metric spaces), and deduce the isoperimetric inequality. This is analogous to Fiala \cite{Fiala}, and Huber \cite{Huber} on two dimensional surfaces with absolutely integrable Gauss curvature.  

\begin{thm}\cite{YW1}
Suppose $(M^n, g)=(\mathbb R^n, e^{2u} |dx|^2)$ is a noncompact complete Riemannian manifold with normal metric. If its $Q$-curvature satisfies
\begin{equation}\int_{M^n}|Q_g|d v_g<\infty
\end{equation}
and 
\begin{equation}\frac{1}{c_n}\int_{M^n} Q_gd v_g<1,
\end{equation}
then there is an $H$-quasiconformal map $f: \mathbb R^n\rightarrow \mathbb R^n$ and constant $C$ such that

\begin{equation}
C^{-1}e^{nu}\leq J_f(x)\leq C e^{nu}\:\: \text{a.e. $x\in\mathbb R^n$},
\end{equation}
where $H$ depends on $n$, and $\alpha$. $C$ depends on the metric $n$ and $g$.

\end{thm}

In dimension 4, the definition of the $Q$-curvature is
\begin{equation}\begin{split}
Q_g=&\frac{1}{12}(-\Delta_g R_g+\frac{1}{4} R_g^2 -3|E_g|^2)\\
=&-\frac{1}{12}\Delta_g R_g+ 2\sigma_2 (A_g),
\end{split}
\end{equation}
where $E_g$ denotes the traceless part of the Ricci curvature.
$$A_g:= \frac{1}{n-2} (Rc_g- \frac{1}{2(n-1)} R_g g)$$ denotes the Schouten tensor.

Thus in dimension 4, the $Q$-curvature differs from $2 \sigma_2(A_g)$ by a divergence term. By a similar argument as in \cite[Lemma 3.2]{LW}, the integral of
 $\Delta_g R_g$ vanishes if the $Q$-curvature is absolutely integrable and the scalar curvature is in $L^{\frac{n}{2}}$ (which is $L^{2}$). 
From this, we can prove that there is a quasiconformal map from this manifold to the Euclidean space, and the dilation $H$ of the quasiconformal map is controlled by the integral of the $ \sigma_2(A_g)$. This is a new phenomenon regarding conformal invariants. Previously, from the work of Theorem 1.1 and 1.2 in \cite{YW1}, we only know that the integral of the $Q$-curvature may control the asymptotic behavior of the conformally flat manifolds. We state this result in the following theorem.
\begin{thm}\label{thm:quasi}
Let $(M^4,g)=(\mathbb R^4,e^{2u} |dx|^2)$ be a noncompact complete conformally flat metric, satisfying  
$$\int_{M^4}|Q_g|d v_g<\infty, $$
and
\begin{equation}\label{R}
\int_{M^4}|R_{g}|^{2}dv_g<\infty,
\end{equation}
If 
\begin{equation}\label{strictSigma}
\dstyle \frac{1}{2\pi^2}\int_{M^4} \sigma_2 (A_g) dv_g< 1
\end{equation}
 then there is an $H$-quasiconformal map $f$ between $(M^4, g)$ and the Euclidean space (with flat metric). The Jacobian $J_f$ of this map satisfies

\begin{equation}
C^{-1}e^{4u}\leq J_f(x)\leq C e^{4u}\:\: \text{a.e. $x\in\mathbb R^4$}.
\end{equation}
Here $C= C( g)$, and $\dstyle H= H( 1-\frac{1}{2\pi^2} \int_{M^n} \sigma_2 (A_g) dv_g)$.
Moreover, $M^4$ satisfies the isoperimetric inequality:
for any smooth bounded domain $\Omega$,
\begin{equation}
Vol_g( \Omega)^{\frac{3}{4}}\leq C  Area_g(\partial\Omega)
\end{equation}
with $C= C( g)$.
\end{thm}

\begin{rem}
Previously we only know that the integral of the $Q$-curvature may control the asymptotic behavior of the conformally flat manifolds. 
Theorem \ref{thm:quasi} indicates that with suitable integrability assumption of the curvature, the value of the integral of the $\sigma_2 (A_g)$ may also control the asymptotic behavior in a similar manner. This includes all the quasiconformal equivalence results and isoperimetric inequality proved in \cite{YW1}.
\end{rem}

\begin{rem}
One can prove a similar result in higher even dimensional manifolds as well. But there are two differences, which would make the statement of the result more complicated. First, in higher dimensions, the condition \eqref{R} needs to be on the Riemannian curvature tensor
$$\int_{M^n} |Rm|^{\frac{n}{2}} dv_g<\infty.
$$
Second, by the Spyros Aleksakis \cite{Alex1, Alex2} classification theorem of global conformal invariants, $\sigma_2 (A_g)$ in \eqref{strictSigma} should be replaced by the Pfaffian of the Riemannian curvature tensor $Pf(Rm)$ (up to a multiplicative constant). For simplicity, we only state the theorem in dimension 4.
\end{rem}

\begin{rem}
Recall that 
$c_n$ denotes the constant $2 ^ {n -2} (\frac {n -2} 2)! \pi ^ {\frac n 2}$.  $c_4 = 4\pi^2$. The constant appeared on the left hand side of \eqref{strictSigma} is equal to $\frac{2}{c_4}$.
\end{rem}

\noindent \textbf{Acknowledgments:} The second author would like to thank Matt Gursky for the question regarding Theorem 1.3, and inspiring discussions during the 2015 Princeton-Tokyo conference on geometric analysis.

\section{Normality of the conformally flat metric}

We only need to prove that the integrability of the negative part of the scalar curvature implies the normality of metric on the end.

\begin{defn}
We call a metric $(\mathbb R^n,e^u\delta_{ij})$ normal if it satisfies
\begin{equation}
u(x)=\frac{1}{c_n}\int_{\mathbb R^n}\log{\frac{|y|}{|x-y|}   }Q_g(y)e^{nu(y)}dy+C
\end{equation}
for some constant $C$ and $c_n$ is a dimensional constant.
\end{defn}

Let $v(x):=\frac{1}{c_n}\int_{\mathbb R^n}\log{\frac{|y|}{|x-y|}   }Q_g(y)e^{nu(y)}dy$ and $h(x):=u(x)-v(x)$. We want to show that $h$ is constant function.

By the conformal transformation of Q-curvature and the fact that $\log\frac{1}{|x|}$ is a fundamental solution of the $\frac{1}{c_n}(-\Delta)^{\frac{n}{2}}$, we have
\begin{equation}\label{1}
(\Delta)^\frac{n}{2}h=(\Delta)^kh=0.
\end{equation}

Moreover, we can use the scalar curvature equation and the integrability condition to get an asymptotic decay of $\Delta h$.

\begin{lem}\label{l1}
With the same assumptions as in Theorem \ref{thm:main}, we have
\begin{equation}\label{2}
\limsup_{r\rightarrow\infty}\fint_{B_r}\Delta h=\limsup_{r\rightarrow\infty}\fint_{B_{2r}\setminus B_r}\Delta h\leq0.
\end{equation}
\end{lem}

\begin{proof}
By the scalar curvature equation of conformally flat manifold, we have
\begin{equation}
\begin{split}
&\int_{B_{2r}\setminus B_r}\Delta u\\
=&\int_{B_{2r}\setminus B_r}-\frac{R_ge^{2u} }{2(n-1) }-\frac{n-2}{2}|\nabla u|^2\\
\leq&\int_{B_{2r}\setminus B_r}\frac{R_g^-e^{2u}}{2(n-1)}\\
\leq& \frac{1}{2(n-1)}(\int_{B_{2r}\setminus B_r}(R_g^-)^{\frac{n}{2}}e^{nu(x)}dx)^{\frac{2}{n}}(\mathrm{Vol_g}(B_{2r}\setminus B_r))^{\frac{n-2}{n}}\\
=&C(n)r^{n-2}.\:(\text{By (\ref{totalR})})
\end{split}
\end{equation}

So
\begin{equation}
\limsup_{r\rightarrow\infty}\fint_{B_{2r}\setminus B_r}\Delta u\leq O(r^{-2})\rightarrow0.
\end{equation}

For $v$, we do integration by part
\begin{equation}
\begin{split}
&|\int_{B_{2r}\setminus B_r}\Delta v|\\
=&|\int_{\partial B_{2r}}\nabla v\cdot\nu-\int_{\partial B_r}\nabla v\cdot\nu |\\
\leq& C(n) (\int_{\partial B_{2r}}|\nabla v|+\int_{\partial B_r}|\nabla v|).
\end{split}
\end{equation}
where $\nu$ is the outer-normal vector field on the boundary.

Notice that
\begin{equation}
\begin{split}
\int_{\partial B_r}|\nabla v|
\leq& (\int_{\partial B_r}|\nabla v|^2d\sigma_r)^{\frac{1}{2}}\cdot(\mathrm{Vol_g}(\partial B_r))^{\frac{1}{2}}\\
=&C(n)r^{\frac{n-1}{2}}(   \int_{\partial B_r}          | \int_{\mathbb R^n} \frac{x-y}{|x-y|^2}Q_{g}e^{nu(y)}dy     |^2               dx    )^{\frac{1}{2}}\\
\leq& C(n)r^{\frac{n-1}{2}}\{  \int_{\partial B_r} [(\int_{\mathbb R^n} |\frac{1}{|x-y|^2}Q_ge^{nu(y)}|dy )(\int_{\mathbb R^n}|Q_g|dv_g)]dx      \}^{\frac{1}{2}}\\
\leq &C(n)r^{\frac{n-1}{2}}\{  \int_{\partial B_r} (\int_{\mathbb R^n} |\frac{1}{|x-y|^2}Q_ge^{nu(y)}|dy )dx      \}^{\frac{1}{2}}\\
= &C(n)r^{\frac{n-1}{2}}\{  \int_{\mathbb R^n}Q_ge^{nu(y)} (\int_{\partial B_r} |\frac{1}{|x-y|^2}|dx )dy      \}^{\frac{1}{2}}\\
\leq &C(n)r^{\frac{n-1}{2}}\{  \int_{\mathbb R^n}Q_ge^{nu(y)} (\int_{\partial B_r,y\in\partial B_r} |\frac{1}{|x-y|^2}|dx )dy      \}^{\frac{1}{2}}\\
= &C(n)r^{\frac{n-1}{2}}\{  \int_{\mathbb R^n}Q_ge^{nu(y)} dy      \}^{\frac{1}{2}} (\int_{\partial B_r,y\in\partial B_r} |\frac{1}{|x-y|^2}|dx )^{\frac{1}{2}}\\
\leq&C'(n)r^{\frac{n-1}{2}} (r^{n-3})^{\frac{1}{2}}\\
=&O(r^{n-2}).
\end{split}
\end{equation}

So
\begin{equation}\label{v}
|\fint_{B_{2r}\setminus B_r}\Delta v|= O(r^{-2})\rightarrow0.
\end{equation}

Combining the above we get
\begin{equation}
\limsup_{r\rightarrow\infty}\fint_{B_{2r}\setminus B_r}\Delta h=\limsup_{r\rightarrow\infty}\fint_{B_{2r}\setminus B_r}\Delta u-\Delta v\leq0.
\end{equation}

\end{proof}

The next lemma was proved in \cite{LW} in dimension 4, but the same argument works in all dimensions.

\begin{lem}\label{l2}
With the same assumptions in Theorem \ref{thm:main}, we have
\begin{equation}
(\fint_{B_{2r}\setminus B_r}|\nabla v|)^2\leq\fint_{B_{2r}\setminus B_r}|\nabla v|^2=O(r^{-2}).
\end{equation}
\end{lem}

\begin{proof}
The first inequality is just H{\"o}lder's inequality. We only need to prove the order of decay for $\fint_{B_{2r}\setminus B_r}|\nabla v|^2$.

\begin{equation}
\begin{split}
\int_{B_{2r}\setminus B_r}|\nabla v|^2
&=\int_{B_{2r}\setminus B_r} | \int_{\mathbb R^n} \frac{x-y}{|x-y|^2}Q_{g}e^{nu(y)}dy     |^2   dx\\
&\leq\int_{B_{2r}\setminus B_r}| \int_{\mathbb R^n} \frac{1}{|x-y|^2}Q_{g}e^{nu(y)}dy     || \int_{\mathbb R^n} Q_{g}e^{nu(y)}dy     |dx\\
&\leq C(n)\int_{B_{2r}\setminus B_r}\int_{\mathbb R^n}|\frac{1}{|x-y|^2}Q_{g}e^{nu(y)}|dydx\\
&= C(n)\int_{\mathbb R^n}\int_{B_{2r}\setminus B_r}|\frac{1}{|x-y|^2}Q_{g}e^{nu(y)}|dxdy\\
&\leq C(n) \int_{\mathbb R^n} |Q_{g}e^{nu(y)}    |( \int_{B_{2r}\setminus B_r}\frac{1}{|x-y|^2}dx) dy\\
&\leq C(n) \int_{\mathbb R^n} |Q_{g}e^{nu(y)}    |( \int_{B_{2r}\setminus B_r,y\in \partial B_r}\frac{1}{|x-y|^2}dx) dy\\
&\leq C(n) \int_{\mathbb R^n} |Q_{g}e^{nu(y)}    | (\int_{B_{3r}}\frac{1}{|x|^2}dx)dy\\
&=O(r^{n-2}).
\end{split}
\end{equation}

Thus by taking average we proved the lemma.

\end{proof}

We're now ready to prove the main theorem. We denote by $\omega_n$ the volume of unit ball in $\mathbb R^n$. The area of the unit sphere in $\mathbb R^n$ is then equal to $n\omega_n$.

\begin{proof} of Theorem \ref{thm:main} From (\ref{1}), we have $\Delta^{k}h=0$ and $\lim_{r\rightarrow\infty}\sup\fint_{B_r}\Delta h\leq0$. As the first step, we will prove $\Delta^{k-1}h=0$.

By the mean value property of harmonic functions, for any $p\in\mathbb R^{2k}$, 
\begin{equation}
\begin{split}
\Delta^{k-1}h(p)
&=\frac{1}{\omega_nr^n}\int_{B_r(p)}\Delta^{k-1}h(x)dx\\
&=\frac{1}{\omega_nr^n}\int_{\partial B_r(p)}\partial_r\Delta^{k-2}h(x)\cdot\nu dx\\
&=\frac{n\omega_n}{\omega_nr}\fint_{\partial B_r(p)}\partial_r\Delta^{k-2}h(x)\cdot\nu dx\\
&=\frac{n}{r}\partial_r\fint_{\partial B_r(p)}\Delta^{k-2}h(x)dx.\\
\end{split}
\end{equation}
Namely
\begin{equation}
\frac{r}{n}\Delta^{k-1}h(p)\leq \partial_r\fint_{\partial B_r(p)}\Delta^{k-2}h(x)dx.
\end{equation}
Integrating the above equation, we get
\begin{equation}
\begin{split}
&\frac{1}{2n}r^2\Delta^{k-1}h(p)+\Delta^{k-2}h(p)\leq\fint_{\partial B_r(p)}\Delta^{k-2}h(x) dx\\
&\frac{\omega_n}{2}r^{n+1}\Delta^{k-1}h(p)+n\omega_nr^{n-1}\Delta^{k-2}h(p)\leq\int_{\partial B_r(p)}\Delta^{k-2}h(x) dx.
\end{split}
\end{equation}
Integrating again both sides, we get
\begin{equation}
\begin{split}
&\frac{\omega_n}{2(n+2)}r^{n+2}\Delta^{k-1}h(p)+\omega_nr^{n}\Delta^{k-2}h(p)\\
\leq&\int_{B_r(p)}\Delta^{k-2}h(x) dx\\
=&\int_{\partial B_r}\nabla_r\Delta^{k-3}h(x)\cdot\nu\\
=&n\omega_nr^{n-1}\fint_{\partial B_r}\nabla_r\Delta^{k-3}h(x)\cdot\nu\\
=&n\omega_nr^{n-1}\partial_r\fint_{\partial B_r}\Delta^{k-3}h(x).\\
\end{split}
\end{equation}
We can rewrite this inequality in the way that
\begin{equation}
\frac{1}{2n(n+2)}r^{3}\Delta^{k-1}h(p)+\frac{1}{n}r\Delta^{k-2}h(p)\leq\partial_r\fint_{\partial B_r}\Delta^{k-3}h(x).\end{equation}
Integrating both sides in $r$, we obtain
\begin{equation}
\frac{1}{8n(n+2)}r^{4}\Delta^{k-1}h(p)+\frac{1}{2n}r^2\Delta^{k-2}h(p)+\Delta^{k-3}h(p)\leq \fint_{\partial B_r}\Delta^{k-3}h(x).
\end{equation}

Keep doing this procedure finitely many times. Then we get
\begin{equation}\label{3}
\begin{split}
&a_{k-1}r^{2(k-2)}\Delta^{k-1}h(p)+a_{k-2}r^{2(k-3)}\Delta^{k-2}h(p)+\cdots+a_1\Delta h(p)\\
\leq&\fint_{\partial B_r}\Delta h(x),
\end{split}\end{equation}
where
\begin{equation}
\begin{split}
a_{k-1}&=\frac{1}{[2\cdot4\cdots(2k-4)]\cdot[2k(2k+2)(2k+4)\cdots(4k-6)]}\\
a_{k-2}&=\frac{1}{[2\cdot4\cdots(2k-6)]\cdot[(2k)(2k+2)\cdots(4k-8)]}\\
&\cdots\\
a_{k-j}&=\frac{1}{[2\cdot4\cdots(2(k-j)-2)]\cdot[(2k)(2k+2)\cdots(4(k-1)-2j)]}\\
&...\\
a_2&=\frac{1}{2\cdot2k}\\
a_1&=1.
\end{split}
\end{equation}
Each $a_i$, for $i=1,\cdots k-1$, is a positive constant.

By (\ref{2}),  the leading cooefficient $a_{k-1}\Delta^{k-1}h(p)$ of the polynomial (\ref{3}) must be non-positive. Since $a_{k-1}>0$, we have
\begin{equation}
\begin{split}
\Delta^{k-1}h(x)=\mathrm{constant}\leq0\\
\Delta^{k-1}\partial_ih(x)=0
\end{split}
\end{equation}
for every $i=1,\cdots,n$.

Applying the mean value property to $\Delta^{k-2}\partial_ih$ and integrating by parts as above, we have
\begin{equation}\label{4}
\begin{split}
&a_{k-1}r^{2(k-2)}\Delta^{k-2}\partial_ih(p)+a_{k-2}r^{2(k-3)}\Delta^{k-3}\partial_ih(p)+\cdots+a_1\partial_i h(p)\\
\leq&\fint_{\partial B_r}\partial_ih(x).
\end{split}
\end{equation}

Next, we make use the scalar curvature equation to see  
\begin{equation}
\begin{split}
&\limsup_{r\rightarrow\infty}\fint_{B_{2r}\setminus B_r}\Delta u+\fint_{B_{2r}\setminus B_r}\frac{n-2}{2}|\nabla u|^2\\
=&\limsup_{r\rightarrow\infty}\fint_{B_{2r}\setminus B_r}-\frac{R_ge^{2u}}{2(n-1)}
\leq0.
\end{split}
\end{equation}

By Holder's inequality, 
\begin{equation}
\begin{split}
&\limsup_{r\rightarrow\infty}\fint_{B_{2r}\setminus B_r}\Delta u+\frac{n-2}{2}(\fint_{B_{2r}\setminus B_r}|\nabla u|)^2\leq0.
\end{split}
\end{equation}

Combine the estimates in Lemma \ref{l1} and Lemma \ref{l2}. Then we have
\begin{equation}\label{5}
\begin{split}
&\limsup_{r\rightarrow\infty}\fint_{B_{2r}\setminus B_r}\Delta h+\frac{n-2}{2}(\fint_{B_{2r}\setminus B_r}|\nabla h|)^2\leq0.
\end{split}
\end{equation}
Thus
\begin{equation}
\begin{split}
&\limsup_{r\rightarrow\infty}\fint_{B_{2r}\setminus B_r}\Delta h+\frac{n-2}{2}(\fint_{B_{2r}\setminus B_r}|\partial_i h|)^2\leq0
\end{split}
\end{equation}
for each $i=1,...,n$.

Plugging (\ref{3}) and (\ref{4}) in (\ref{5}), we have a polynomial $P_i(r)$  satisfying
\begin{equation}
\limsup_{r\rightarrow\infty}P_i(r)\leq0.
\end{equation}
So the leading coefficient of $P_i(r)$ must be non-positive. Namely
\begin{equation}
a_{k-1}^2|\Delta^{k-2}\partial_ih(p)|^2\leq0
\end{equation}
for each $p\in\mathbb R^n$. 

This implies that
\begin{equation}\label{intermediate}
\Delta^{k-2}\partial_ih=0
\end{equation}
for each $i=1,...,n$, and thus
\begin{equation}
\Delta^{k-1}h=0.
\end{equation}

Now we reduce the problem to
\begin{equation}
\begin{split}
&\Delta^{k-1}h=0\\
&\text{with}\\
&\limsup_{r\rightarrow\infty}\fint_{B_r}\Delta h\leq0
\end{split}
\end{equation}

As an intermediate step, we obtain (\ref{intermediate}) $\Delta^{k-2}\partial_ih=0$.

We can apply this argument inductively to get
\begin{equation}
\begin{split}
\Delta^{k-3}\partial_ih=0,\:\text{for each $i$}\\
\Delta^{k-2}h=0.
\end{split}
\end{equation}
After finitely many steps, we get
\begin{equation}\label{6}
\Delta\partial_ih=0
\end{equation}
for each $i=1,...,n$, and thus
\begin{equation}
\Delta^2h=0.
\end{equation}

Notice that the induction argument will make use of (\ref{2}), so we cannot obtain $\Delta h=0$ directly from the induction.

However, since $\Delta h$ is harmonic, by Liouville's theorem and (\ref{2}), we have
\begin{equation}
\Delta h=C_0\leq0.
\end{equation}

Argue as in \cite{CQY1}
\begin{equation}\label{7}
\begin{split}
&\limsup_{r\rightarrow\infty}\fint_{B_{2r}\setminus B_r}|\nabla h|^2\\
\leq &\limsup_{r\rightarrow\infty}\fint_{B_{2r}\setminus B_r} (2|\nabla u|^2+2|\nabla v|^2)\\
= &\limsup_{r\rightarrow\infty}\fint_{B_{2r}\setminus B_r} -\frac{4}{n-2}\Delta u-\frac{2R_g^-e^{2u}}{(n-1)(n-2)}+2|\nabla v|^2\\
=& \limsup_{r\rightarrow\infty}\fint_{B_{2r}\setminus B_r} -\frac{4}{n-2}\Delta h \\
&+ \limsup_{r\rightarrow\infty}\fint_{B_{2r}\setminus B_r}  \frac{4}{n-2}\Delta v -\frac{2R_g^-e^{2u}}{(n-1)(n-2)}+2|\nabla v|^2\\
\leq &-\frac{4}{n-2}C_0+0\\
=&-\frac{4}{n-2}C_0.
\end{split}
\end{equation}
In the last inequality we used Lemma \ref{l1}, Lemma \ref{l2} and the integrability of the negative part of scalar curvature.

(\ref{6}) tells us that $\partial_i h$ is harmonic for each $i=1,...,n$. By Liouville's theorem again, we have
\begin{equation}
\partial_i h=\mathrm{constant},
\end{equation}
and thus
\begin{equation}
\Delta h=0.
\end{equation}
$C_0$ can be chosen to be $0$. So $\partial_i h=0$ for each $i=1,...,n$.
Therefore we have proved that $h$ is a constant. This concludes the proof that the metric is normal.

Now that the metric is normal. By Theorem 1.1 and Theorem 1.3 of \cite{CQY1}, we have (\ref{totalQ}).
\end{proof}

\section{Quasiconformal map and the isoperimetric inequality}

\begin{proof} of Theorem \ref{thm:quasi}
In this theorem, we use the $L^{2}$-integrability of the scalar curvature in two ways. We first need it to apply Theorem \ref{thm:main} to show that the metric is normal. Also we need it to prove that 
the divergence term in the $Q$-curvature vanishes.

Since \eqref{R} holds, Theorem \ref{thm:main} implies that the metric is normal. 
Suppose in addition that the strict inequality holds in \eqref{totalQ}, i.e. 
\begin{equation}\label{strictQ}
\dstyle \frac{1}{c_4}\int_{M^4} Q_g dv_g <  \chi(\mathbb{R}^{4} )=1.
\end{equation}
Then it was already proved by the second author in \cite{YW1} that there exists an $H$-quasiconformal map between $M^4$ and the Euclidean space with Jacobian comparable to the volume form $e^{4u}$.
Thus the problem reduces to prove \eqref{strictQ}.
For this purpose, we first see that 
\begin{equation}
Q_g=-\frac{1}{12}\Delta_g R_g+ 2\sigma_2 (A_g).
\end{equation}
In \cite[Lemma 3.2] {LW}, it was shown that
if the metric is normal and the second fundamental form $L$ of the isometric embedding $M^4 \hookrightarrow \mathbb R^5$ satisfies
\begin{equation}\label{L}
\int_{M^4}|L|^{4}dv_g<\infty,
\end{equation}
then 
\begin{equation}\label{R2} \int_{M^4} \Delta_g R_g dv_g =0.\end{equation}
We can adopt a similar method to prove \eqref{R2}
under the assumption \eqref{R} 
\begin{equation*}
\int_{M^4}|R_g|^{2}dv_g<\infty.
\end{equation*}
For completeness we give the proof of \eqref{R2} in the following.

\begin{lem}\label{vanishing}
Let $(M^4,g)=(\mathbb R^4,e^{2u} |dx|^2)$ be a noncompact complete conformally flat metric, satisfying  
$$\int_{M^4}|Q_g|d v_g<\infty, $$
\begin{equation}
\int_{M^4}|R_{g}|^{2}dv_g<\infty,
\end{equation}
and the metric is normal on each end.
Then $$\int_{M^4} \Delta_g R_g dv_g=0.$$
\end{lem}

\begin{proof}[Proof of Lemma \ref{vanishing}]
Let $B^0(0, \rho)$ be the ball centered at the origin, with radius $\rho$ with respect
to the Euclidean metric. On the Euclidean space, there always
exists a smooth cut-off function $\eta_\rho$ which is supported on $B^0(0, 2\rho)$.
It is equal to $1$ on $B^0(0, \rho)$, and its $k$-th derivative
is of order $O(1/\rho^k)$ over the annulus $B^0(0, 2\rho)\setminus B^0(0, \rho)$.
Again since the $Q$-curvature is absolutely integrable, so is $\Delta_gR_g$.

Since $\eta_\rho=1$ on $B^0(0,\rho) $,
\begin{equation}\label{2.3}
\begin{split}
&\int_{M^4}\Delta_gR_g d v_g\\
=&\lim_{\rho\rightarrow\infty} \int_{  B^0(0, 2\rho) }\Delta_gR_g \eta_\rho d v_g\\
=&\lim_{\rho\rightarrow\infty} \int_{B^0(0, 2\rho)\setminus B^0(0, \rho)}R_g \Delta_g \eta_\rho d v_g.\\
\end{split}\end{equation}
Here the last equality holds because all boundary terms in the integration by parts formula vanish,
and $\Delta_g\eta_\rho=0$ on the complement of $B^0(0, 2\rho)\setminus B^0(0, \rho)$.

Using
$$ dv_g=e^{4w}dx,$$
$$\Delta_g\eta_\rho dv_g=\partial_i(e^{2w}\partial_i\eta_\rho) dx, $$ we have
\begin{equation} \begin{split}
&\int_{B^0(0, 2\rho)\setminus B^0(0, \rho)}  R_g \Delta_g\eta_\rho dv_g\\
&=
\int_{B^0(0, 2\rho)\setminus B^0(0, \rho)}  R_g \partial_i(e^{2w}\partial_i\eta_\rho )dx \\
=& \int_{B^0(0, 2\rho)\setminus B^0(0, \rho)}  R_g (\Delta_0\eta_\rho e^{2w}+ \partial_i(e^{2w})\partial_i\eta_\rho )dx \\
\leq&  C\int_{B^0(0, 2\rho)\setminus B^0(0, \rho)}   \frac{R_g}{\rho^2} e^{2w}dx\\
&   + C \int_{B^0(0, 2\rho)\setminus B^0(0, \rho)}  \frac{R_g |\partial_i w|}{\rho}e^{2w}dx \\
=:& I+II.\\
\end{split} \end{equation}
The first term $I$ can be bounded by the $L^2$-norm of the scalar curvature.
\begin{equation}\begin{split}
| I|\leq&C (\int_{B^0(0, 2\rho)\setminus B^0(0, \rho)} | R_g|^2 e^{4w}dx)^{1/2}\cdot
(\int_{B^0(0, 2\rho)\setminus B^0(0, \rho)} \frac{1}{\rho^4}dx)^{1/2}\\
&\rightarrow 0,\\
\end{split}\end{equation}
as $\rho$ tends to $\infty$.

We will now study $II$ through the asymptotic behavior of the derivatives of $w$. We notice that the pointwise estimate of
$\partial_i w$ is not known. But since we are taking the integral over the annulus (with respect to the Euclidean metric),
it can be reduced to the integral estimate of $\partial_i w$ over spheres at the end of the manifold.
\begin{equation}\begin{split}
|II|=&C \left|\int_{B^0(0, 2\rho)\setminus B^0(0, \rho)}  \frac{R_g \partial_i w}{\rho}e^{2w}dx\right|\\
\leq&C (\int_{B^0(0, 2\rho)\setminus B^0(0, \rho)} | R_g|^2 e^{4w}dx)^{1/2}\cdot
(\int_{B^0(0, 2\rho)\setminus B^0(0, \rho)} \frac{|\partial_i w|^2}{\rho^2}dx)^{1/2}.\\
\end{split}\end{equation}
Notice that
\begin{equation}\begin{split}\label{3.1}
&  \int_{B^0(0, 2\rho)\setminus B^0(0, \rho)} |\partial_i w|^2dx\\
=&\int_{B^0(0, 2\rho)\setminus B^0(0, \rho)}\left|\frac1{4\pi^2}\int_{\mathbb R^4}\frac{x_i-y_i}{|x-y|^2}Q e^{4w(y)} d y\right|^2 d v_0\\
\leq&C \int_{B^0(0, 2\rho)\setminus B^0(0, \rho)} \left| \int_{\mathbb{R}^4}\frac{1}{|x-y|}Q(y)e^{4w(y)} dy\right|^2dx\\
\leq&C \int_{B^0(0, 2\rho)\setminus B^0(0, \rho)}\int_{\mathbb{R}^4}\frac{1}{|x-y|^2}  Q(y)e^{4w(y)} dydx \cdot \int_{\mathbb{R}^4} Q(y)e^{4w(y)} dy.\\
\end{split}\end{equation}
Since for any $y\in \mathbb{R}^4$, we have
$$\int_{x\in \partial B^0(0, r)} \frac{1}{|x-y|^2} d\sigma(x)=| \partial B^0(0, r)|\cdot O(\frac{1}{r^2}),$$
\begin{equation} \begin{split}\int_{B^0(0, 2\rho)\setminus B^0(0, \rho)}\frac{1}{|x-y|^2} dx=&\int_\rho^{2\rho}
\int_{x\in \partial B^0(0, r)} \frac{1}{|x-y|^2} d\sigma(x)dr\\
=&\int_\rho^{2\rho}| \partial B^0(0, r)|\cdot O(\frac{1}{r^2})
dr=O(\rho^2).\\\end{split}
\end{equation}
Plugging this into (\ref{3.1}), and using the fact that  $\int_{\mathbb{R}^4}Q(y)e^{4w(y)} dy<\infty$,
we obtain
\begin{equation}\begin{split}
\int_{B^0(0, 2\rho)\setminus B^0(0, \rho)} |\partial_i w|^2dx
\leq C (\int_{\mathbb{R}^4}Q(y)e^{4w(y)} dy)^2 \cdot O(\rho^2)= O(\rho^2).
\end{split}\end{equation}
Therefore,
\begin{equation}\begin{split}
| II|\leq&C(\int_{B^0(0, 2\rho)\setminus B^0(0, \rho)} | R|^2 dv_g)^{1/2}\cdot
(\frac{1}{\rho^2}\int_{B^0(0, 2\rho)\setminus B^0(0, \rho)} |\partial_i w|^2dx)^{1/2}\\
\leq&C(\int_{B^0(0, 2\rho)\setminus B^0(0, \rho)} | R|^2 dv_g)^{1/2}\rightarrow 0\\
\end{split}\end{equation}
as $\rho$ tends to $\infty$.
To conclude,
\begin{equation}\begin{split}
|\int_{M^4} \Delta_g R_g dv_g|=&\lim_{\rho\rightarrow\infty}| \int_{B^0(0, 2\rho)\setminus B^0(0, \rho)} R_g \Delta_g \eta_\rho dv_g| \\
\leq& \lim_{\rho\rightarrow\infty} (|I|+|II|)=0.\\  \end{split}\end{equation}

This completes the proof of the lemma.
\end{proof}


From this, and inequality \eqref{strictSigma}, we deduce \eqref{strictQ}. Therefore we complete the proof of the existence of a quasiconformal map. 
The isoperimetric inequality is a direct consequence of the existence of such an $H$-quasiconformal map.
\end{proof}

\end{document}